\documentclass[12pt,reqno]{amsart}
\usepackage[pagewise]{lineno} 

\usepackage[ansinew]{inputenc}
\usepackage{graphicx}
 
\usepackage[bookmarksnumbered,colorlinks, linkcolor=blue, citecolor=red, pagebackref, bookmarks, breaklinks]{hyperref}
 
\newtheorem{thm}{Theorem}[section]
\newtheorem{cor}[thm]{Corollary}
\newtheorem{lem}[thm]{Lemma}
\newtheorem{prop}[thm]{Proposition}

\theoremstyle{definition}

\theoremstyle{remark}

\numberwithin{equation}{section}

\def\diver{\mathop{\text{\normalfont div}}}

\def\dist{\mathop{\text{\normalfont dist}}}

\usepackage[left=3cm,right=3cm,top=2.9cm,bottom=2.9cm]{geometry}


\begin{document}

\title{Lower bounds for fractional Orlicz-type eigenvalues	}

\author[ A. Salort]{  Ariel   Salort}
 
\address{A. Salort \newline Departamento de Matem\'aticas y Ciencia de Datos, Universidad San Pablo-CEU, CEU Universities, Urbanizaci\'on Montepr\'incipe, 28660 Boadilla del Monte, Madrid, Spain. }
\email{{\tt amsalort@gmail.com, ariel.salort@ceu.es}\hfill\break\indent {\it Web page:} \url{https://sites.google.com/view/amsalort/}}

\maketitle

\begin{abstract}
In this article, we establish precise lower bounds for the eigenvalues and critical values associated with the fractional $A-$Laplacian operator, where $A$ is a Young function. The obtained bounds are expressed in terms of the domain geometry and the growth properties of the function $A$. We emphasize that we do not assume that $A$ or its complementary function satisfies the $\Delta_2$ condition.
\end{abstract}

\section{Introduction}
One of the central problems in the analysis of the $
p-$Laplacian type operators is the study of its eigenvalues, which are closely related to the structure of the underlying domain and the boundary conditions imposed. In particular, the first eigenvalue 
$$
\lambda_1=\inf\left\{\int_\Omega |\nabla u|^p\,dx \text{ for  } u\in C_c^\infty(\Omega) \text{ such that }\int_\Omega \omega(x)|u|^p \,dx =1\right\}
$$
related to the nonlinear problem defined for $p>1$ as
\begin{align} \label{ec.i1}
\begin{cases}
-\diver(|\nabla u|^{p-2}\nabla u) =\lambda \omega |u|^{p-2}u &\text{ in }\Omega\\ u=0 &\text{ on }\partial\Omega
\end{cases}
\end{align}
has been extensively studied, as it provides important information about the behavior of solutions to the geometry of the domain. Here $\omega$ is a suitable weight function and $\Omega\subset \mathbb{R}^n$ denotes an open and bounded set. See for instance \cite{L1,L2}.

While upper bounds for eigenvalues have been established in a variety of settings, obtaining sharp lower bounds remains a challenging and active area of research. Lower bounds are of particular importance because they offer insights into the stability and regularity of solutions, as well as estimates for the oscillatory behavior of eigenfunctions.

In the one-dimensional case  with a weight function, lower bounds were obtained in \cite{E, LYHA,P,Pi, SL}. When $\Omega\subset \mathbb{R}^n$, $n\geq 2$ several results are known. In \cite{AL,C},  lower bounds in terms of the measure of the domain were obtained. Indeed,  when $\omega\in L^\theta(\Omega)$ satisfies
\begin{equation*}
\frac{C}{|\Omega|^{\frac{p}{n}-\frac{1}{\theta}} \|\omega\|_{L^\theta(\Omega)} } \leq \lambda_1\quad \text{ where } \theta=\begin{cases} \theta_0\in (\frac{n}{p},\infty] &\text{ when } 1<p< n\\ 1  &\text{ when }n<p,
\end{cases}
\end{equation*}
where $C=C(n,p)>0$.
In addition, in \cite{DNP, Pi}, more accurate bounds involving the inner radius $r_\Omega$ of $\Omega$ are obtained as a consequence of Lyapunov type inequalities:
\begin{align} \label{ineq.s1}
\frac{C}{r_\Omega^{p-\frac{n}{\theta}} \|\omega\|_{L^\theta (\Omega)} }&\leq \lambda_1 \quad
\text{ where } \theta=\begin{cases} \theta_0\in (\frac{n}{p},\infty] &\text{ when } 1<p< n\\ 1  &\text{ when }n<p,
\end{cases}
\end{align}
where $C=C(n,p)>0$ and $r_\Omega:=\max\{\dist(x,\partial\Omega)\colon x\in\Omega\}$ is the inradius of $\Omega$. Later, in \cite{JKS} these results were extended to the nonlocal case obtaining that when $\omega \in L^\theta(\Omega)$
\begin{align} \label{ineq.s1}
\frac{C}{r_\Omega^{sp-\frac{n}{\theta}} \|\omega\|_{L^\theta (\Omega)} }&\leq \lambda_1^s \quad
\text{ where } \theta=\begin{cases} \theta_0\in (\frac{n}{sp},\infty] &\text{ when } 1<sp< n\\ 1  &\text{ when }n<sp,\end{cases}
\end{align}
where here $\lambda_1^s$ is the first eigenvalue related to the fractional $p-$Laplacian operator of order $s\in (0,1)$.

When operators follow a growth  more general than a power law, the concept of eigenvalue becomes highly dependent on the normalization of the eigenfunction due to the potential lack of homogeneity. More precisely, equation \eqref{ec.i1} can be generalized by replacing the power $p$ with a so-called Young function: given a Young function $A$, and a bounded domain $\Omega\subset \mathbb{R}^n$, $n\geq 1$, consider the problem
\begin{align} \label{eq.a.1}
\begin{cases}
-\diver(a(|\nabla u|) \frac{\nabla u}{|\nabla u|}) = \lambda \omega a(|u|)\frac{u}{|u|} & \text{ in } \Omega\\
u=0 &\text{ on } \partial\Omega
\end{cases}
\end{align}
where $\lambda\in \mathbb{R}$ is the eigenvalue parameter, $\omega$ is a suitable weight function and $a(t)=A'(t)$, $t>0$. Observe that \eqref{eq.a.1} boils down to \eqref{ec.i1} when $A(t)=t^p$, $p>1$.

\noindent In this case, in \cite{eig1, eig2} it is  proved that given $\alpha>0$ there exists a critical value $\lambda_{1,\alpha}>0$ and a function  $u_\alpha$ such that $\int_\Omega \omega A(|u_\alpha|)\,dx=\alpha$ and $\frac{1}{\alpha}\int_\Omega A(|\nabla u_\alpha|)\,dx = \lambda_{1,\alpha}$. From this, it can be deduced the existence of an eigenvalue $\Lambda_{1,\alpha}$ with corresponding eigenfunction $u_\alpha$ in the sense that pair $(\Lambda_{1,\alpha}, u_\alpha)$ satisfies \eqref{eq.a.1} in the weak sense. The quantities $\Lambda_1$ and $\lambda_1$ are in general different and coincide  only when $A$ is homogeneous.

A first result concerning the lower bounds of \eqref{eq.a.1} can be found in \cite{pdnliapu}. In the one-dimensional case, assuming that $A$ satisfies the $\Delta_2$ condition (that is, there exists $c\geq 1$ such that $A(2t)\leq c A(t)$ for any $t>0$), the authors establish that for any $\alpha>0$,
$$
\frac{C_{p}}{\|\omega\|_{L^1(a,b)}}\leq \lambda_{1,\alpha} 
$$
where $\Omega=(a,b)\subset \mathbb{R}$, and $p>1$ is defined as $\lim_{r\to\infty}\frac{A(rt)}{A(r)}=t^{p-1}$. Similar bounds in the one-dimensional case were found in \cite{vergara} when $A$ is a  submultiplicative Young function (that is, there exists $c\geq 1$ such that $A(rt)\leq c A(r)A(t)$ for any $r,t\geq 0$).

\noindent When $\Omega\subset \mathbb{R}^n$, $n\geq 1$ and $A$ is a submultiplicative Young function, in Theorems 4.4 and 4.2 of \cite{S22} it is proved that given $\alpha>0$ there exists a computable constant $C>0$ independent of $\alpha$, depending only of $A$ and $n$ such that
\begin{align*}
\frac{1}{\alpha}\left[A\left(  \frac{ C \sigma(r_\Omega)}{A^{-1}(\|\omega\|_{L^1(\Omega)}^{-1})}\right)\right]^{-1} \leq \lambda_{1,\alpha} \quad &\text{ when } \omega\in L^1(\Omega) \text{ and } \sigma(1)<\infty,\\
\frac{1}{\alpha}\left[ A\left( \frac{C}{ A^{-1}\left(\frac{1}{\tau_A(\Omega)\|\omega\|_{L^\infty(\Omega)} }\right)   }  \right) \right]^{-1}\leq \lambda_{1,\alpha} \quad &\text{ when } w\in L^\infty(\Omega) \text{ and  } \sigma(1)=\infty
\end{align*} 
where $\sigma(t)=\int_{t^{-n}}^\infty  A^{-1}(r)r^{-(1+\frac1n)}\,dr$ and $\tau_A(\Omega):=|\Omega| (\tilde A)^{-1}(|\Omega|^{-1})$, being $\tilde A$ the complementary function of $A$. These inequalities, in the case $A(t)=t^p$, $p>1$, recover the corresponding inequalities in \eqref{ineq.s1}. 

In the last years nonlocal operator with a non-standard growth have received an increasing attention and an active community is currently working on problems involving operators defined in terms of a Young function $A(t)=\int_0^t a(\tau)\,d\tau$ having the form
$$
(-\Delta_a)^s u (x) = \text{p.v.}\, \int_{\mathbb{R}^n} a(|D^s u|)\frac{D^s u}{|D^s u|}\frac{dy}{|x-y|^n},
$$
where $s\in(0,1)$,  $D^su(x,y) = \frac{u(x)-u(y)}{|x-y|^s}$ and  p.v. stand for \emph{principal value}. This nonhomogeneous operator represents a generalization of the fractional $p-$Laplacian of order $s\in (0,1)$.
We refer to \cite{ACPS2, ACPS,ACPS24, ABC, BS21, BKO23,  C2,C1,FBSp24, FBPLS,FBS19,  FBSV2, FBSV23,  S20, S22,SV22} and references therein. In particular, the nonlocal version of problem \eqref{eq.a.1} takes the form
\begin{align}\label{eq1.intro}
\begin{cases}
(-\Delta_a)^s u= \lambda \omega\frac{a(|u|)}{|u|} u &\quad \text{ in } \Omega\\
u=0& \quad \text{ on } \partial\Omega.
\end{cases}
\end{align}
where $\lambda\in \mathbb{R}$ is the eigenvalue parameter and $\omega$ is a suitable weight function.
We refer to \cite{srati,  BKO23,choi,FBSp24,  FBSV23,S20,S22, SV22} for properties and results related with the nonlocal nonstandard growth eigenvalue problem \eqref{eq1.intro}.

As in the local case, the non-homogeneity of the problem makes the eigenvalue highly dependent of the normalization of the eigenfunction. More precisely, in \cite{S20, SV22} it is  proved that given $\alpha>0$ there exists a critical value $\lambda_{1,\alpha}^s>0$ and $u_\alpha^s$ such that $\int_\Omega \omega A(|u^s_\alpha|)\,dx=\alpha$ and $\frac{1}{\alpha}\int_\Omega A(|D^s u^s_\alpha|)\,dx = \lambda^s_{1,\alpha}$. More precisely, we consider
\begin{equation} \label{minimi.intro}
\lambda_{1,\alpha}^s = \inf\left\{\frac{1}{\alpha}\iint_{\mathbb{R}^{2n}} A(|D^s  u|)\, d\nu_n \colon u\in C_c^\infty(\Omega), \int_\Omega  \omega A(|u|)\, dx = \alpha\right\},
\end{equation}
where $d\nu_n = |x-y|^{-n}dxdy$ and $\omega$ is a suitable weight function.

\noindent From this, it can be deduced the existence of an eigenvalue $\Lambda_{1,\alpha}^s$ with corresponding eigenfunction $u_\alpha^s$ in the sense that pair $(\Lambda^s_{1,\alpha}, u^s_\alpha)$ satisfies \eqref{eq1.intro} in the weak sense, being $\Lambda_1^s$ and $\lambda_1^s$ different  when $A$ is  inhomogeneous, but comparable each other when $A$ satisfies the doubling condition.

To the best of our knowledge, estimates for  $\Lambda_{1,\alpha}^s$ and $\lambda_{1,\alpha}^s$ have not been previously studied in the literature. The goal of this article is to establish lower bounds for these quantities.

An important aspect of analyzing \eqref{eq1.intro} is whether the Young function $A$ satisfies a so-called \emph{doubling condition}.  This condition is crucial for controlling constants within the function:
\begin{itemize}
\item[$\circ$] $A$ satisfies the \emph{doubling condition near infinity} (denoted as $A\in \Delta_2^\infty$) if there exists $C_\infty\geq 2$ such that $A(2t)\leq C_\infty A(t)$ for all $t\geq T_\infty$,
\item[$\circ$] $A$ satisfies the \emph{doubling condition near zero} (denoted as $A\in \Delta_2^0$) if there exists $C_0\geq 2$ such that $A(2t)\leq C_0 A(t)$ for all $t\leq  T_0$.
\item[$\circ$] $A$ satisfies the \emph{global doubling condition} (denoted as $A\in \Delta_2$) if the previous condition and fulfilled, and it is denoted $\Delta_2=\Delta_2^0\cap \Delta_2^\infty$.
\end{itemize}
Assuming or relaxing the doubling condition introduces significant technical challenges in the analysis, such as the potential loss of reflexivity in the associated fractional Orlicz-Sobolev spaces. Moreover, imposing this condition on either the function $A$ or its conjugate $\tilde A$ is known to yield both upper and lower bounds for the corresponding Young function in terms of power functions. For further details, see Section \ref{sec.young}.

To characterize  the growth of a general Young function $A $  (which may not satisfy the doubling condition), we use the Matuszewska-Orlicz functions associated with $A$, along with the corresponding Matuszewska-Orlicz indexes, defined as follows:
$$
M_A(t)=\sup_{\alpha>0} \frac{A(\alpha t)}{A(\alpha)}, \qquad M_0(t,A)=\liminf_{\alpha\to 0^+} \frac{A(\alpha t)}{A(\alpha)},\qquad 
M_\infty(t,A)=\liminf_{\alpha\to\infty} \frac{A(\alpha t)}{A(\alpha)}
$$
$$
i (A)= \lim_{t\to\infty} \frac{\log M_A (t)}{\log t}, \qquad i_0 (A)= \lim_{t\to\infty} \frac{\log M_0 (t,A)}{\log t}, \qquad i_\infty (A)= \lim_{t\to\infty} \frac{\log M_\infty (t,A)}{\log t}.
$$
See Section \ref{sec.matu} for details and the precise definitions.

\medskip
We now present our main results. We emphasize that, unless explicitly stated otherwise, we do not assume the $\Delta_2$ condition on $A$ or its complementary function $\tilde A$.

\begin{thm}  \label{teo1.intro}
Let $s\in (0,1)$ and let $A$ be a Young function satisfying \eqref{cond1}. Let $\Omega\subset \mathbb{R}^n$ be an open bounded domain with inner radius $r_\Omega$. Given  $\omega\in L^1(\mathbb{R}^n)$ and $\alpha>0$, consider the critical value $\lambda_{1,\alpha}^s$ defined in \eqref{minimi.intro}. The following holds:
\begin{itemize}
\item[i)] There exists a unique $\alpha_0>0$ satisfying the equation
$$
\alpha_0 \lambda_{1,\alpha_0}^s = r_\Omega^n.
$$

\item[ii)] 
Assume that $i_0(A)>\frac{n}{s}$ when $\alpha\leq \alpha_0$ or $i_\infty(A) >\frac{n}{s}$ when $\alpha> \alpha_0$. Then, there exists a positive, computable constant $C=C(n,s,A)$ such that
\begin{equation*} 
\frac{C}{\|\omega\|_{L^1(\Omega)}} \frac{r_\Omega^n}{M_A(r_\Omega^s)} \le \lambda_{1,\alpha}^s.
\end{equation*}

\noindent In particular, this holds when $i_0(A)>\frac{n}{s}$ if $\alpha\ll1$ or when $i_\infty(A) >\frac{n}{s}$ if $\alpha\gg1$.
\end{itemize}

\end{thm}
 
In Section \ref{ejemplos}, we compute the Matuszewska-Orlicz functions and indices for several notable Young functions. With this,  we state Theorem \ref{teo1.intro} for some interesting cases:

\begin{itemize}
\item[i)] When $A(t)=t^p$, $p>1$ the eigenvalue problem becomes homogeneous, then for any $\alpha>0$, when $sp>n$:
$$
\frac{1}{\|\omega\|_{L^1(\Omega)}} \frac{C}{r_\Omega^{sp-n}} \le \lambda_{1,\alpha}^s
$$
which in some extent  recovers \eqref{ineq.s1}.

\medskip

\item[ii)]  Given $1<p<q<\infty$, consider $A(t)=\frac{t^p}{p} + \frac{t^q}{q}$. Then $A\in \Delta_2$. This gives the eigenvalue problem for the fractional $p$-$q-$Laplacian (see for instance \cite{ambrosio}).  When $\alpha\ll1$  and $sp>n$, or $\alpha\gg1$ and $sq>n$, then 
$$
\frac{1}{\|\omega\|_{L^1(\Omega)}} \frac{C}{\max\{r_\Omega^{sp-n},r_\Omega^{sq-n}\}} \le \lambda_{1,\alpha}^s.
$$

\medskip

\item[iii)]  Given $p,q,r\geq 1$, consider $A(t)=t^p \ln^r (1+t^q)$.  Then $A\in \Delta_2$. When $\alpha\ll1$  and $s(p+qr)>n$, or $\alpha\gg1$ and $sp>n$, then 
$$
\frac{1}{\|\omega\|_{L^1(\Omega)}} \frac{C}{\max\{r_\Omega^{s(p-qr)-n},r_\Omega^{sp-n}\}} \le \lambda_{1,\alpha}^s.
$$
\medskip

\item[iv)]  For $k\in\mathbb{N}$ define $\displaystyle A(t)=e^t - \sum_{j=0}^{k-1} \frac{t^j}{j!}$. Then $A\in \Delta_2^0$ but $A\notin \Delta_2^\infty$. When $\alpha\ll1$ and $sk>n$, when $r_\Omega \leq 1$ we have that
$$
\frac{1}{\|\omega\|_{L^1(\Omega)}} \frac{C}{r_\Omega^{sk-n}} \le \lambda_{1,\alpha}^s.
$$

\medskip

\item[v)] Consider the function $A(t)=e^{e^t}- e$. Then $A\in \Delta_2^0$ but $A\notin \Delta_2^\infty$. When $\alpha\gg1$ it holds that
$$
\frac{1}{\|\omega\|_{L^1(\Omega)}} \frac{C}{r_\Omega^{s-n}} \le \lambda_{1,\alpha}^s.
$$
\end{itemize}

In Theorem \ref{teo2}, we also  derive a lower bound for the minimizer using the inverse of the Young function instead of the Matuszewska-Orlicz function. Moreover, in Corollary \ref{coro.1} we prove that the same lower bounds hold for the eigenvalue $\Lambda_{1,\alpha}^s$ when $A\in \Delta_2$.

\begin{thm} \label{teo2.intro}
Assume that $\Omega\subset\mathbb{R}^n$ is a bounded domain with diameter $d_\Omega$ containing the origin. Let $s\in (0,1)$ and let $A$ be a Young function satisfying conditions \eqref{cond2} and $i(A)<\frac{n}{s}$. Given  $\omega\in L^\infty(\mathbb{R}^n)$ and $\alpha>0$, consider $\lambda_{1,\alpha}^s$ as in \eqref{minimi.intro}. Then, there exists a positive constant $C=C(n,s,A)$ such that
$$
\frac{C}{\|\omega\|_{L^\infty(\Omega)} M_A(d_\Omega^s) }\leq \lambda_{1,\alpha}^s.
$$
\end{thm}

To apply Theorem \ref{teo2.intro}, there is an implicit growth condition: the condition $i(A)<\frac{n}{s}$ is not satisfied when   $A\not\in \Delta_2^0$ or $A\not \in \Delta_2^\infty$, (see Lemma \ref{lema.m} for details). Therefore, this result can be applied only when $A\in \Delta_2$.

Under the  assumption of the $\Delta_2$ condition on $A$, we improve Theorem \ref{teo2.intro} by replacing the diameter with the inner radius.

\begin{thm} \label{teo4.intro}
Let $s\in (0,1)$ and let $A\in \Delta_2$ be a Young function satisfying  \eqref{cond3}. Assume that $\Omega\subset \mathbb{R}^n$ is a bounded Lipschitz domain with inner radius $r_\Omega$.  Given  $\omega\in L^\infty(\mathbb{R}^n)$ and $\alpha>0$, consider  $\lambda_{1,\alpha}^s$ as in \eqref{minimi.intro}.  Then, there exists a positive constant $C=C(n,s,A)$  such that
$$
\frac{C}{  \|\omega\|_{L^\infty(\Omega)} M_A( r_\Omega^s) }\leq \lambda_{1,\alpha}^s.
$$
\end{thm}

\noindent Here we state some notable examples derived from Theorem \ref{teo4.intro}.  

\begin{itemize}
\item[i)] When $A(t)=t^p$, $p>1$, this gives the eigenvalue problem for the fractional $p-$Laplacian, which  is homogeneous. Then for any $\alpha>0$, when $sp<n$:
$$
\frac{1}{\|\omega\|_{L^\infty(\Omega)}}  \frac{C}{ r_\Omega^{sp}}\leq \lambda_{1,\alpha}^s
$$
which in some extent  recovers \eqref{ineq.s1}.

\medskip

\item[ii)]  Given $1<p<q<\infty$, consider $A(t)=\frac{t^p}{p} + \frac{t^q}{q}$. Then $A\in \Delta_2$. This gives the eigenvalue problem for the fractional $p-q-$Laplacian (see for instance \cite{ambrosio}). Then, for $\alpha>0$, when $sq<n$, it holds that
 $$
 \frac{1}{\|\omega\|_{L^\infty(\Omega)}} \frac{C}{ \max\{r_\Omega^{sq}, r_\Omega^{sp}\}}\leq \lambda_{1,\alpha}^s.
 $$
 
 \medskip
  
 \item[iii)] Given $p,q,r\geq 1$,  consider $A(t)=t^p \ln^r (1+t^q)$.  Then $A\in \Delta_2$. Then, for $\alpha>0$, when  $s(p+qr)<n$, it holds that
 $$
 \frac{1}{\|\omega\|_{L^\infty(\Omega)}} \frac{C}{ \max\{r_\Omega^{sp}, r_\Omega^{s(p+qr)}\}}\leq \lambda_{1,\alpha}^s.
 $$
 
 \end{itemize} 
In particular, Theorem \ref{teo2.intro} establishes these same inequalities but with the diameter in place of the inner radius under the same hypothesis on the parameters.

The same lower bounds established  in Theorems \ref{teo2.intro} and \ref{teo4.intro} also hold for the eigenvalue $\Lambda_{1,\alpha}^s$, as stated in Corollary \ref{coro.2}.

\section{Preliminaries}

\subsection{Young functions} \label{sec.young}
A function $A\colon [0,\infty)\to [0,\infty]$ is called a \emph{Young function} if it is convex, non-constant, left-continuous and $A(0)=0$. A function with these properties admits the representation 
$$
A(t)=\int_0^t a(\tau)\,d\tau \qquad \text{ for } t\geq 0,
$$
for some non-decreasing, left-continuous function $a\colon[0,\infty)\to [0,\infty]$. 

The \emph{complementary function} $\tilde A$ of $A$ is the Young function defined as
$$
\tilde A (t)= \sup\{ \tau t-A(\tau)\colon \tau \geq 0\} \qquad \text{ for } t\geq 0.
$$
One has that
$$
t \leq A^{-1}(t) (\tilde A)^{-1}(t) \le 2t \qquad \text{ for } t\geq 0.
$$

\noindent From the convexity of the Young function it is immediate that
\begin{equation} \label{prop.conv}
A(rt)\leq r A(t)\quad  \text { for } 0<r<1, \qquad 
A(rt)\geq r A(t)\quad  \text { for } r>1.
\end{equation}
Moreover, from the integral representation of the Young function it follows that
$$
G(2t) > tg(t), \qquad G(t)\leq t g(t).
$$

\subsubsection{The doubling condition}
A Young function $A\in \Delta_2^\infty$ (or $A\in \Delta_2^0$) if and only if there exists $p>1$ and $T_\infty>0$ (or $T_0>0$) such that
\begin{equation} \label{cte.p}
\frac{ta(t)}{A(t)} \leq p \quad \text{ for all } t\geq T_\infty \quad (\text{or } 0<t\leq T_0).
\end{equation}

\noindent It is easy to see that 

\noindent $A\in \Delta_2^\infty$ if there exists $C_\infty\geq 2$ such that $A(2t)\le C_\infty A(t)$ for all $t\geq T_\infty$.

\noindent $A\in \Delta_2^0$ if there exists $C_0\geq 2$ such that
$A(2t) \leq C_0 A(t)$ for all $t\leq T_0$.

\medskip
\noindent We define $\Delta_2 = \Delta_2^0\cap \Delta_2^\infty$. The following statements are equivalent:
\begin{itemize}
\item[i)] $A\in \Delta_2$
\item[ii)] there exists $p>1$ such that $\frac{ta(t)}{A(t)} \le p$ for all $t>0$
\item[iii)] there exists $C\geq 2$ such that $A(2t) \leq C A(t)$ for all $t>0$.
\end{itemize}

\begin{prop}
Let $A$ be a Young function such that $A\in \Delta_2^0$ and let $p>1$ be the number defined in \eqref{cte.p}. Then
$$
\tau^p A(t) \leq A(t\tau)\leq 	\tau A(t) \quad \text{ for } 0<\tau <1 \text{ and } t<T_0.
$$
Similarly, if $A\in \Delta_2^\infty$ we have that
$$
\tau  A(t) \leq A(t\tau)\leq 	\tau^p A(t) \quad \text{ for } \tau >1 \text{ and } t>T_\infty.
$$
Hence, if $A\in \Delta_2$, then
$$
\min\{\tau,\tau^p\} A(t)\leq A(t\tau) \leq \max\{\tau, \tau^p\} A(t) \quad \text{ for } \tau>0.
$$
\end{prop}

\subsubsection{Ordering  of functions}
A Young function $A$ dominates another Young function $B$ near infinity if there exists a positive constant $c$ and $t_0$ such that
$$
B(t)\leq A(ct) \qquad \text{ for }t\geq t_0.
$$
The functions $A$ and $B$ are called \text{equivalent near infinity} if they dominate each other in the respective range of values of their argument, and we write $A \simeq B$.

$A\approx B$ means that $A$ and $B$ are bounded by each other, up to a multiplicative constant.

\subsection{Matuszewska indexes} \label{sec.matu}
The Matuszewska-Orlicz functions associated to the Young function $A$ are defined as follows:
$$
M(t,A)=\sup_{\alpha>0} \frac{A(\alpha t)}{A(\alpha)}, \qquad M_0(t,A)=\liminf_{\alpha\to 0^+} \frac{A(\alpha t)}{A(\alpha)},\qquad 
M_\infty(t,A)=\liminf_{\alpha\to\infty} \frac{A(\alpha t)}{A(\alpha)}.
$$
These functions are nondecreasing, submultiplicative in the variable $t$ and equal to one at $t=1$. We also consider the Matuszewska-Orlicz indices at zero and infinity, respectively,  defined as
$$
i (A)= \lim_{t\to\infty} \frac{\log M (t,A)}{\log t}, \qquad i_0 (A)= \lim_{t\to \infty} \frac{\log M_0 (t,A)}{\log t}, \qquad i_\infty (A)= \lim_{t\to\infty} \frac{\log M_\infty (t,A)}{\log t}.
$$

It follows from the definitions that $M_0(t,A)\leq M(t,A)$ and $M_\infty(t,A)\leq M(t,A)$, and then $\max\{i_0(A), i_\infty(A)\} \leq i(A)$.

When there is no confusion, we will remove the dependence on $A$.

It is easy to see that the Matuszewska function can be bounded in terms of powers if and only if $A,\tilde A\in \Delta_2$, that is,  
\begin{equation} \label{cota.M}
	\min\{ t^{p^+_A},t^{p^-_A}\}\leq M(t,A)  \leq \max\{t^{p^-_A},t^{p^+_A} \}
\end{equation}
where $p^+_A = \sup_{t>0} \frac{a(t)t}{A(t)}$ and $p^-_A = \inf_{t>0} \frac{a(t)t}{A(t)}$ .

For a comprehensive approach on these functions and indices we refer to the monograph \cite{M89}.

\subsection{Examples of Young functions} \label{ejemplos}

Here we provide for some examples of Young functions and compute their corresponding Matuszewska functions and indexes. For further examples we refer to \cite{M89}.

\medskip

\noindent {\bf Example 1}. Let $p>1$, and assume that 
$$
A(t) \simeq t^p \quad \text{ when } t\ll1.
$$
Then $A\in \Delta_2^0$. In this case we have that $M_0(t)=t^p$ and $i_0(A)=p$. If we assume that
$$
A(t)\simeq t^p \quad \text{ when } t\gg1,
$$
then $A\in \Delta_2^\infty$, $M_\infty(t)=t^p$ and $i_\infty(A)=p$.

\noindent In particular, when $A(t)=t^p$, $M(t,A)=M_0(t,A)=M_\infty(t,A)=t^p$ and $i(A)=i_0(A)=i_\infty(A)=p$.
As a special case, if $1<p<q<\infty$, 
$$
A(t)=\frac{t^p}{p} + \frac{t^q}{q}
$$ 
then 
\begin{align*}
M_0(t,A)=t^{p}, \quad M_\infty(t,A)= t^q, \quad M(t,A)=\max\{t^p, t^q\},
\end{align*}
and $i_0(A)=p$, $i_\infty(A)=i(A)=q$.

\medskip

\noindent {\bf Example 2}. Let $r\geq 0$ and $p\geq 1$, then if 
$$
A(t)\simeq t^p \ln^r t \quad \text{ when } t\ll1
$$
then $A\in \Delta_2^0$ and in this case, $M_0(t) = t^{p}$ and $i_0(A)=p$. If  
$$
A(t)\simeq t^p \ln^r t \quad \text{ when } t\gg1, 
$$
then $M_\infty(t)=t^p$ and $i_\infty(A)=p$.

\noindent As a special case, if $p,q,r\geq 0$ and $A(t)=t^p \ln^r (1+t^q)$ then we have that
\begin{align*}
M_0(t,A)=t^{p+qr}, \quad M_\infty(t,A)= t^p, \quad M(t,A)=\max\{t^p, t^{p+qr}\}
\end{align*}
and  $i(A)=i_0(A)=p+qr$, $i_\infty(A)=p$.

\medskip

\noindent {\bf Example 3}. For $k\in\mathbb{N}$ define $\displaystyle A(t)=e^t - \sum_{j=0}^{k-1} \frac{t^j}{j!}$. Then $A\in \Delta_2^0$ but $A\notin \Delta_2^\infty$, and
\begin{align*}
M_0(t)=t^k, \qquad 
M_\infty(t)=
\begin{cases}
0 &\text{ if } 0<t<1\\
\infty &\text{ if } t>1,
\end{cases}  \qquad 
M(t)=
\begin{cases}
t^k &\text{ if } 0<t\leq 1\\
\infty &\text{ if } t>1,
\end{cases}
\end{align*}
$i_0(A)= k$, $i(A)=i_\infty(A)=\infty$.

\medskip

\noindent {\bf Example 4}. For $r>0$, assume that $A(t)\simeq e^{-t^{-r}}$ for $t\ll1$. Then $A\notin \Delta_2^0$ and 
\begin{align*}
M_0(t)=
\begin{cases}
0 & \text{ if } 0<t<1\\
1 & \text{ if } t=1\\
\infty &\text{ if } t>1.
\end{cases}
\end{align*}
In particular, when $A(t)=e^{-t^{-r}}$,  we have that
\begin{align*}
M_\infty(A)=1, \qquad M_0(t)=
\begin{cases}
0 & \text{ if } 0<t<1\\
1 & \text{ if } t=1\\
\infty &\text{ if } t>1,
\end{cases}
\qquad M(t)=
\begin{cases}
1 & \text{ if } 0<t\leq1\\
\infty &\text{ if } t>1.
\end{cases}
\end{align*}
and
$i_\infty(A)=i(A)=0$, $i_0(A)=\infty$.

\noindent {\bf Example 5}. Assume that $A(t)\simeq e^{e^t}$  for $t\gg1$. Then $A\notin \Delta_2^\infty$ and 
\begin{align*}
M_\infty(t)=
\begin{cases}
0 & \text{ if } 0<t<1\\
1 & \text{ if } t=1\\
\infty &\text{ if } t>1.
\end{cases}
\end{align*}
In particular, when $A(t)=e^{e^t}- e$, since $A(t)\simeq et$ when $t\ll1$, 
\begin{align*}
M_0(t)=t, \qquad 
M_\infty(t)=
\begin{cases}
0 & \text{ if } 0<t<1\\
1 & \text{ if } t=1\\
\infty &\text{ if } t>1,
\end{cases}
\qquad 
M(t)=
\begin{cases}
t & \text{ if } 0<t\leq 1\\
\infty &\text{ if } t>1,
\end{cases}
\end{align*}
and $i_0(A)=1$, $i(A)=i_\infty(A)=\infty$.

\begin{lem} \label{lema.m}
Let $A$ be a Young function such that $A\notin \Delta_2^k$ for $k=0$ or $k=\infty$. Then 
\begin{align*}
M(t)=
\begin{cases}
1 & \text{ if } t= 1\\
\infty & \text{ if } t>1.
\end{cases}
\end{align*}
When $0<t<1$ it holds that $M(t)\leq t$. Moreover, $i(A)=\infty$.
\end{lem}
\begin{proof}
First, observe that from from \eqref{prop.conv}, it holds that $M(t) \leq t\quad$  for  $0<t<1$.

In light of \cite[Proposition 2.1]{fbs.autov}, if $A\notin \Delta_2^0$ or $A\notin \Delta_2^\infty$ then we have that
\begin{align*}
M_k(t)=
\begin{cases}
1 & \text{ if } t=1\\
\infty & \text{ if } t>1,
\end{cases}
\end{align*}
$k=0,\infty$, respectively. Moreover, by definition, we have that $M_0(t)\leq M(t)$ and  $M_\infty(t)\leq M(t)$ for any $t>0$. This gives immediately that when $A\notin \Delta_2^k$ for $k=0$ or $k=\infty$, one has that $M(1)=1$ and $M(t)=\infty$ when $t>1$. 

\noindent Finally, since $M(t)=\infty$ for $t>1$, this gives that $i(A)=\infty$.
\end{proof}

\begin{lem}
If $A\in \Delta_2$ then there exists $p\geq 1$ such that $M(t)=t^p$.
\end{lem}
\begin{proof}
Observe that since $A\in \Delta_2$, then there exists $q>1$ such that $A(r t)\leq \max\{t,t^q\} A(r)$ for any $t,r\geq 0$. Then for any $t>0$ it holds that $M$ is finite:
$$
M(t)=\sup_{\alpha>0} \frac{A(\alpha t)}{A(\alpha)} \leq \max\{t,t^{q}\}.
$$
Moreover, observe that
$$
M(rt)=
\sup_{\alpha>0} \frac{A(t r\alpha )}{A(r\alpha)} \frac{A(r\alpha)}{A(\alpha)} \leq 
\sup_{\alpha>0} M(t) \frac{A(r\alpha)}{A(\alpha)} \leq M(t)M(s)
$$
and $M(1)=1$, that is, $M(t)$ is submultiplicative.

Define $v(t)=\ln(M(e^t))$. This function is additive, that is, $v(r+t)=v(r)+v(t)$ for any $s,t\in\mathbb{R}$. It is well known that measurable additive functions are linear, therefore, there exists $p\in \mathbb{R}$ such that $v(t)=pt$ from there $M(t)=t^p$. Finally, from \eqref{prop.conv}, 
\begin{equation*} 
M(t) \leq t\quad  \text{ for } 0<t<1 \quad \text{ and } \quad M(t)\geq t\quad  \text{ for } t>1.
\end{equation*}
which implies that $p\geq 1$. 
\end{proof}

\subsection{Some useful inequalities}
Given $s\in(0,1)$ and a Young function $A$ such that
\begin{equation} \label{cond1}
\int^\infty \left(\frac{t}{A(t)}\right)^\frac{s}{n-s}\,dt<\infty
\end{equation}
consider the Young function $E$ given by
\begin{equation} \label{func.e}
E(t)=t^\frac{n}{n-s}\int_t^\infty \frac{\tilde A(\tau)}{\tau^{1+\frac{n}{n-s}}}\,d\tau\quad \text{ for } t\geq 0.
\end{equation}
Moreover, consider $\Psi_s\colon (0,\infty)\to (0,\infty)$ defined as
$$
\Psi_s(r)=\frac{1}{r^{n-s} E^{-1}(r^{-n})}, \qquad \text{ for } r>0.
$$

The following properties of the function $\Psi_s$ are studied in \cite[Proposition 2.1]{ACPS24}:
\begin{lem} \label{lema1}
Let $s\in (0,1)$ and let $A$ be a Young function. Assume \eqref{cond1}. Then

\begin{itemize}
  \setlength\itemsep{0.5em}
\item[(i)] The function $\Psi_s$ is non-decreasing.

\item[(ii)] Define the Young function $B=\tilde E$. Then
$
\Psi_s(r)\approx r^s B^{-1}(r^{-n}) \text{ for } r>0.
$

\item[(iii)] If $i_\infty(A)>\frac{n}{s}$ then $B\simeq A$ near infinity, and
$
\Psi_s(r) \approx  r^s A^{-1}(r^{-n})  \text{ for } 0<r\leq 1.
$

\item[(iv)] If $i_0(A)>\frac{n}{s}$ then $B\simeq A$ near $0$, and
$
\Psi_s(r) \approx  r^s A^{-1}(r^{-n}) \text{ for } r\geq 1.
$ 
\end{itemize}
 
\end{lem}

The following modular Morrey type inequality is proved in \cite[Remark 4.3]{ACPS24}:
\begin{prop} \label{morrey}
Let $s\in (0,1)$ and let $A$ be a Young function satisfying  \eqref{cond1}. Then, $W^s L^A(\mathbb{R}^n)\subset C^{\Psi_s(\cdot)}(\mathbb{R}^n)$. Moreover,  for any $u\in W^s L^A(\mathbb{R}^n)$ and $x,y\in \mathbb{R}^n$ it holds that
$$
|u(x)-u(y)|\leq C_M |x-y|^s B^{-1}\left( \frac{1}{|x-y|^n} \iint_{\mathbb{R}^{2n}} A(D^s u(z,w))\,d\nu_n(z,w) \right)
$$
for some constant $C_M$ depending on $n$ and $s$, where the Young function $B$ is given by $B(t)=\tilde E(t)$, being $E$ the Young function defined in \eqref{func.e}.
\end{prop}

\noindent The following Hardy type inequality is proved in Theorem 5.1 and Proposition C in \cite{ACPS2}. Let $A$ satisfy the following conditions
\begin{equation} \label{cond2}
\int^\infty \left(\frac{t}{A(t)}\right)^\frac{s}{n-s}\,dt=\infty, \qquad 
\int_0 \left(\frac{t}{A(t)}\right)^\frac{s}{n-s}\,dt<\infty.
\end{equation}

\begin{prop} \label{hardy}
Let $s\in (0,1)$ and let $A$ be a Young function satisfying conditions \eqref{cond2} and $i(A)<\frac{n}{s}$. Then for all $u\in C_c^\infty(\mathbb{R}^n)$  it holds that
$$
 \int_{\mathbb{R}^n} A\left(C_{H_1}\frac{|u(x)|}{|x|^s}\right)\,dx \leq (1-s) \iint_{\mathbb{R}^{2n}} A(C_{H_2} |D^s u|) \,d\nu_n
$$
for positive constants $C_{H_1}$ and $C_{H_2}$ depending only on $n$ and $s$.
\end{prop}

Given a bounded domain $\Omega\subset \mathbb{R}^n$, we denote $\delta_\Omega(x):=\inf\{|x-y|\colon y \in \Omega^c\}$ the distance from $x$ to $\partial \Omega$. 

The following Hardy type inequality is proved in \cite[Theorem 1.5]{BKS}.
\begin{prop} \label{hardy2}
Let $s\in (0,1)$ and let $\Omega\subset \mathbb{R}^n$ be a bounded Lipschitz  domain. If $A\in \Delta_2$ and
\begin{equation} \label{cond3}
\lim_{k \to \infty}  \sup_{t\geq 0} \frac{A(k t)}{k^\frac{n}{s}A(t)}=0 =  \lim_{k\to 0^+}  \sup_{t\geq 0} \frac{A(k t)}{k^\frac1s A(t)},
\end{equation}
then, there exists a positive constant $C_{H_3}$ such that for all $u\in C^\infty_c(\Omega)$ 
$$
 \int_{\Omega} A\left( \frac{|u(x)|}{\delta_\Omega(x)}\right)\,dx \leq C_{H_3} \iint_{\mathbb{R}^{2n}} A(|D^s u|) \,d\nu_n.
$$

\end{prop}

\subsection{Orlicz and Orlicz-Sobolev spaces}
As we mentioned in the introduction of this section, the main reference for Orlicz spaces is the book \cite{KR}. As for Orlicz-Sobolev spaces, the reader can consult, for instance, with \cite{Gossez}.    Fractional order Orlicz-Sobolev spaces, as we will use them here, were introduced in \cite{FBS19} and then further analyze by several authors. The results used in this paper can be  found in \cite{ACPS2, ACPS, FBSp24}.

\subsubsection{Orlicz spaces}
Given a bounded domain $\Omega\subset \mathbb{R}^n$  and a Young function $A$, the \emph{Orlicz class} is defined as
$$
\mathcal{L}^A(\Omega) :=\left\{u\in L^1_\text{loc}(\Omega)\colon \int_\Omega A(|u|)\, dx<\infty\right\}.
$$ 
The \emph{Orlicz space} $L^A(\Omega)$ is defined as the linear hull of $\mathcal{L}^A(\Omega)$ and is characterized as
$$
L^A(\Omega) = \left\{u\in L^1_\text{loc}(\Omega)\colon \text{ there exists } k>0 \text{ such that }  \int_\Omega A\left(\frac{|u|}{k}\right)\, dx<\infty\right\}.
$$
In general the Orlicz class is strictly smaller than the Orlicz space, and $\mathcal{L}^A(\Omega) = L^A(\Omega)$ if and only if $A\in \Delta_2^\infty$. The space $L^A(\Omega)$ is a Banach space when it is endowed, for instance, with the {\em Luxemburg norm}, i.e.
$$
\|u\|_{L^A(\Omega)} = \|u\|_A :=\inf\left\{k>0\colon  \int_\Omega A\left(\frac{|u|}{k}\right)\, dx\le 1\right\}.
$$
\noindent This space $L^A(\Omega)$ turns out to be separable if and only if $A\in \Delta_2^\infty$.

An important subspace of $L^A(\Omega)$ is $E^A(\Omega)$ that it is defined as the closure of the functions in $L^A(\Omega)$ that are bounded. This space is characterized as
$$
E^A(\Omega) = \left\{u\in L^1_\text{loc}(\Omega)\colon  \int_\Omega A\left(\frac{|u|}{k}\right)\, dx<\infty \text{ for all } k>0\right\}.
$$
\noindent This subspace $E^A(\Omega)$ is separable, and we have the inclusions
$$
E^A(\Omega)\subset \mathcal{L}^A(\Omega)\subset L^A(\Omega) 
$$
with equalities if and only if $A\in \Delta_2^\infty$.
Moreover, the following duality relation holds
$$
(E^A(\Omega))^* = L^{\tilde A}(\Omega),
$$
where the equality is understood  via the standard duality pairing. Observe that this automatically implies that $L^A(\Omega)$ is reflexive if and only if $A, \tilde A\in \Delta_2^\infty$.

\subsubsection{Fractional Orlicz-Sobolev spaces}
Given a fractional parameter $s\in (0,1)$, we define the \emph{H\"older quotient} of a function $u\in L^A (\Omega)$ as
$$
D^su(x,y) = \frac{u(x)-u(y)}{|x-y|^s}.
$$
Then, the \emph{fractional Orlicz-Sobolev space} of order $s$ is defined as
$$
W^sL^A(\mathbb{R}^n) := \{u\in L^A(\mathbb{R}^n)\colon D^su\in L^A(\mathbb{R}^{2n}, d\nu_n)\},
$$
where $d\nu_n = |x-y|^{-n}dxdy$ and
$$
W^sE^A(\mathbb{R}^n) := \{u\in E^A(\mathbb{R}^n)\colon D^su\in E^A(\mathbb{R}^{2n}, d\nu_n)\}.
$$
When $A\in \Delta_2$, these spaces coincide and we denote
$$
W^{s, A}(\mathbb{R}^n) =W^sL^A(\mathbb{R}^n)= W^sE^A(\mathbb{R}^n).
$$
The space $W^s L^A(\mathbb{R}^n)$ is reflexive if and only if $A, \tilde  A\in \Delta_2$.

In these spaces the norm considered is
$$
\|u\|_{W^sL^A(\mathbb{R}^n)} = \|u\|_{s, A} = \|u\|_A + [u]_{s,A,\mathbb{R}^n}
$$
with
$$
[u]_{s,A,\mathbb{R}^n}=\inf\left\{k >0 \colon \iint_{\mathbb{R}^{2n}} A\left(\frac{|D^s u(x,y)|}{k} \right)\,d\nu_n \leq 1\right\}.
$$
\noindent Again, with this norm, $W^sL^A(\mathbb{R}^n)$ is a Banach space and $W^sE^A(\mathbb{R}^n)$ is a closed subspace. The space $W^s_0 L^A(\Omega)$ is then defined as the closure of $C^\infty_c(\Omega)$ with respect to the topology $\sigma(W^sL^A(\mathbb{R}^n), W^sE^{\tilde A}(\mathbb{R}^n))$ and $W^s_0 E^A(\Omega)$ as the closure of $C^\infty_c(\Omega)$ in norm topology. 

\section{Eigenvalues and critical points}
Let $A$ be a Young function, and let $\Omega\subset \mathbb{R}^n$ be an open and bounded set. For a fixed normalization parameter $\alpha>0$, we define the {\em critical point} $\lambda_{1,\alpha}^s$ as
\begin{equation} \label{minimi}
\lambda_{1,\alpha}^s = \inf\left\{\frac{1}{\alpha}\iint_{\mathbb{R}^{2n}} A(|D^s  u|)\, d\nu_n \colon u\in C_c^\infty(\Omega), \int_\Omega  \omega A(|u|)\, dx = \alpha\right\}.
\end{equation}
Here, $\omega$ is a suitable positive weight function. We  assume  that  $\omega\in L^1(\mathbb{R}^n)$ when \eqref{cond1} holds, and $\omega \in L^\infty(\mathbb{R}^n)$ when \eqref{cond2} holds.

In \cite{SV22} (see also \cite{S20} when $A\in \Delta_2$) it is proved that \eqref{minimi} is solvable, that is,  there exists a \emph{minimizer} $u_\alpha^s\in W^{s}_0 L^A(\Omega)$ such that $\int_\Omega \omega A(|u_\alpha^s|)\,dx =\alpha$ and
\begin{equation} \label{eq}
\iint_{\mathbb{R}^{2n}} A(|D^s u_\alpha^s|)\,d\nu_n = \lambda_{1,\alpha}^s \int_\Omega \omega A(|u_\alpha^s|)\,dx.
\end{equation}
By applying an appropriate version of the Lagrange multipliers theorem, we can establish the existence of an eigenvalue $\Lambda_{1,\alpha}^s$ with  corresponding eigenfunction $u_\alpha^s$. More precisely,  $u_\alpha$ is a weak solution to the following equation with $\Lambda=\Lambda_{1,\alpha}^s$
\begin{align}\label{eq1}
\begin{cases}
(-\Delta_a)^s u= \Lambda\omega  \frac{a(|u|)}{|u|} u &\quad \text{ in } \Omega\\
u=0& \quad \text{ on } \partial\Omega,
\end{cases}
\end{align}
being $a(t)=A'(t)$ for $t>0$, and where the fractional $a-$Laplacian of order $s\in (0,1)$ is the nonlocal and non-standard growth operator defined as
$$
(-\Delta_a)^s u (x) = \text{p.v.}\, \int_{\mathbb{R}^n} a(|D^s u|)\frac{D^s u}{|D^s u|}\,d\nu_n,
$$
that is, for all $v\in C^\infty_c(\Omega)$ it holds that
$$
\iint_{\mathbb{R}^{2n}} a(|D^s u_\alpha^s|)\frac{D^s u_\alpha^s D^s v}{|D^s u_\alpha^s|}\, d\nu_n = \Lambda_{1,\alpha}^s \int_\Omega \omega a(|u_\alpha^s|) \frac{u_\alpha^s v}{|u_\alpha^s|}\, dx.
$$
\noindent 
We refer also to \cite{SBO, FBSp24} for the existence of higher order eigenvalues.

\begin{lem} \label{lema.comp}
Let $A$ be a Young function such that $A'(t)=a(t)$ for any $t\geq 0$. Then
$$
\frac{1}{p_A}\Lambda_{1,\alpha}^s\leq \lambda_{1,\alpha}^s \leq   p_A \Lambda_{1,\alpha}^s.
$$
The number $p_A:=\sup_{\beta>0} \frac{a(\beta)\beta}{A(\beta)}$ is finite if and only if $A\in \Delta_2$.
\end{lem}
\begin{proof}
Given $\alpha>0$, consider the critical point  $\lambda_{1,\alpha}^s$ associated to the minimizing function $u_\alpha^s$ such that $\int_\Omega A(|u_\alpha|) \,dx =\alpha$. Observe that
\begin{align} \label{desii1}
\begin{split}
\int_\Omega \omega A(|u_\alpha^s|)\,dx &\geq \inf_{\beta>0} \frac{A(\beta)}{a(\beta)\beta}\int_\Omega \omega a(|u_\alpha^s|)|u_\alpha^s|\,dx 
=
\frac{1}{p_A}\int_\Omega \omega a(|u_\alpha^s|)|u_\alpha|\,dx.
\end{split}
\end{align}
Moreover, since $a(t)=A'(t)$ is increasing,  $A(t)=\int_0^t a(\tau)\,d\tau \leq a(t)t$ for any $t>0$. This fact, together with \eqref{desii1} gives that
$$
\lambda_{1,\alpha}^s = \frac{\int_\Omega A(|D^s u_\alpha^s|)\,d\nu_n}{\int_\Omega \omega A(|u_\alpha^s|)\,d\nu_n} \leq \frac{\int_\Omega a(|D^s u_\alpha^s|) |D^s u_\alpha^s|\,d\nu_n}{\frac{1}{p_A}\int_\Omega \omega a(|u_\alpha^s|) |u_\alpha^s|\,dx} = p_A \Lambda_{1,\alpha}^s.
$$
The other bound is analogous. Finally, note that $p_A<\infty$ if and only if $A\in \Delta_2$.
\end{proof}

For example,  the number $p_A$ as defined in Lemma \ref{lema.comp}, takes the following form for the following notable Young functions $A\in \Delta_2$.

\begin{itemize}
\item[i)] Let $p>1$ and $A(t)=t^p$ then $p_A=p$.
 
\item[ii)] Let $p,q>1$, $r\geq 0$ and consider $A(t)=\frac{t^p}{p}+ \frac{t^q}{q}$, then $p_A=q$ when $t\geq 1$ and $p_A=p$ when $t<1$.

\item[iii)] Let $p,q>1$, $r\geq 0$ and consider $A(t)=\frac{t^p}{p} \ln^r(1+t^q)$, then $p_A=p+qr$.
 \end{itemize}

\noindent Let us state now some useful relations.
\begin{lem} \label{lema.e}
For $\alpha>0$ let the function $E(\alpha):=\alpha\lambda_{1,\alpha}^s$. Then $E$ is strictly positive, strictly increasing, $E(0)=0$, $E(\infty)=\infty$ and $E$ is a Lipschitz function for $\alpha>0$.

In particular, we have that $\lambda_{1,1}\leq   \alpha \lambda_{1,\alpha}^s$ when $\alpha>1$ and $\alpha \lambda_{1,\alpha}^s\leq \lambda_{1,1}$ when $\alpha<1$.
\end{lem}
\begin{proof}
First, given $\beta>0$ and a fixed function $u\in C_c^\infty(\Omega)$ such that $\int_\Omega A(|u_\beta|)\,dx =\beta$, since $A$ is continuous and nondecreasing, if we define the function $\phi\colon \mathbb{R}_+ \to \mathbb{R}_+$ by
$$
\phi(r)=\int_\Omega A(r|u_\beta|)\,dx,
$$
it follows that $\phi$ is continuous, nondecreasing, $\phi(0)=0$, $\phi(1)=\beta$ and $\phi(\infty)=\infty$. Hence, for any $\alpha>0$ there exists $r_\alpha>0$ such that $\phi(r_\alpha)=\alpha$ and in particular
\begin{equation} \label{r.alpha.1}
r_\alpha<1 \text{ when } \alpha < \beta, \qquad r_\alpha>1 \text{ when } \alpha>\beta,
\end{equation}
\begin{equation} \label{r.alpha}
r_\alpha\to 0 \text{ when }\alpha\to 0, \qquad  r_\alpha\to \infty \text{ when }\alpha\to\infty.
\end{equation}

Let us see that $E$ is strictly increasing. Let $0<\alpha<\beta$ and in light of \eqref{eq}, let $u_\beta\in W^s_0 E^A(\Omega)$ be such that 
$$
\int_\Omega A(|u_\beta|)\,dx=\beta, \qquad \lambda_{1,\beta}^s = \frac{1}{\beta}\iint_{\mathbb{R}^{2n}} A(|D^s u_\beta|)\,d\nu_n.
$$
By \eqref{r.alpha.1} there exists $r_\alpha<1$ such that $\int_\Omega A(r_\alpha |u_\beta|)\,dx = \alpha$. Therefore, using the convexity of $A$ we obtain the desired inequality:
$$
\alpha\lambda_{1,\alpha}^s \leq \iint_{\mathbb{R}^{2n}} A(r_\alpha |D^s u_\beta|)\,d\nu_n \leq r_\alpha \iint_{\mathbb{R}^{2n}} A(|D^s u_\beta|)\,d\nu_n <  \beta\lambda_{1,\beta}^s.
$$
Moreover, from the previous inequality together with \eqref{r.alpha} we get that
$$
0\leq \lim_{\alpha\to 0^+}E(\alpha) \leq \lim_{\alpha\to 0^+} r_\alpha \iint_{\mathbb{R}^{2n}} A(|D^s u_\beta|)\,d\nu_n = E(\beta)  \lim_{\alpha\to 0^+} r_\alpha =0,
$$
from where $E(0)=0$. $E(\alpha)$ is lower semicontinuous by \cite[Lemma 4.3]{fbs.autov}, and then $\liminf_{\alpha\to \infty} E(\alpha) \geq \infty$.
 Finally, $E$ is Lipschitz continuous by Theorem 4.5 in \cite{fbs.autov}.
\end{proof}

\begin{prop} \label{autov.compara}
Let $A$ be a Young function and let $\Omega\subset \mathbb{R}^n$ be open and bounded,  let $B_1\subset \mathbb{R}^n$ be a ball  such that $|\Omega|=|B_1|$ and let $B_2\subset \Omega$ be a ball. Then
$$
\lambda_{1,\alpha}^s(B_1) \leq \lambda_{1,\alpha}^s(\Omega) \leq \lambda_{1,\alpha}^s(B_2).
$$
\end{prop}
\begin{proof}
Let $u\in W^s_0 L^A(\Omega)$ be such that $\int_\Omega A(|u|)\,dx =\alpha$. Denote by $u^*$ the symmetric rearrangement of $u$. Thus, $u^*$ is radially decreasing about 0 and is equidistributed with $u$. Using the P\'olya-Szeg\"o principle stated in \cite[Theorem 3.1]{ACPS} we get that
$$
\iint_{\mathbb{R}^{2n}} A(|D^s u^*|)\,d\nu_n \leq \iint_{\mathbb{R}^{2n}} A(|D^s u|)\,d\nu_n.
$$
Hence, if $B_1$ is a ball having the same measure of $\Omega$, since $\int_{B_1} A(|u^*|)\,dx = \alpha$, we get
$$
\lambda_{1,\alpha}^s(B_1) \leq \lambda_{1,\alpha}^s(\Omega).
$$

On the other hand, consider a ball $B_2\subset \Omega $ and the function $u_{B_2} \in W^{s}_0L^A(B_2)$  such that $\int_B A(|u_{B_2}|)\,dx =\alpha$  to be a minimizer for $\lambda_{1,\alpha}^s(B_2)$. Define $v\in W^{1,A}_0(\Omega)$ defined as the extension of $u_{B_2}$ by zero outside $\Omega$. Then $\int_\Omega A(|v|)\,dx=\int_\Omega A(|u_{B_2}|)\,dx =\alpha$ and therefore $v$ is admissible in the variational characterization of $\lambda_{1,\alpha}^s(\Omega)$. Hence
\begin{align*}
\lambda_{1,\alpha}^s(\Omega) &\leq \frac{1}{\alpha}\iint_{\mathbb{R}^{2n}} A(|D^s v|)\,d\nu_n = \lambda_{1,\alpha}^s(B_2),
\end{align*}
which concludes the proof.
\end{proof}

\section{Lower bounds of critical values and eigenvalues}

We start this section with the proof of Theorem \ref{teo1.intro}.

\begin{proof}[Proof of Theorem \ref{teo1.intro}]
Fixed $\alpha>0$, let $u_\alpha^s \in W^s_0 L^A(\Omega)$ be a minimizer of \eqref{minimi} such that $\int_\Omega \omega A(|u_\alpha^s|)\,dx =\alpha$, i.e., the pair $(u_\alpha^s, \lambda_{1,\alpha}^s)$ satisfies equation \eqref{eq}, where  $\lambda_{1,\alpha}^s$ is defined in \eqref{minimi}. Since $s\in(0,1)$ is fixed, for simplicity we will drop the dependence in $s$.

In light of Proposition \ref{morrey},  $u_\alpha$ is continuous and hence there exists $x_0\in \Omega$ such that
$$
|u_\alpha(x_0)|=\max\{|u_\alpha(x)|\colon x \in \mathbb{R}^n\}>0.
$$
From Proposition \ref{morrey} we have that for any $x,y\in \mathbb{R}^n$ it holds that
$$
|u_\alpha(x)-u_\alpha(y)|\leq C_M |x-y|^s B^{-1}\left(  \frac{1}{|x-y|^n} \iint_{\mathbb{R}^{2n}} A(|D^s u_\alpha|) \,d\nu_n\right),
$$
where the Young function $B$ is the complementary function of the Young function defined in \eqref{func.e}. We take $x=x_0$, $y\in \partial\Omega$, from where the previous expression becomes
$$
|u_\alpha(x_0)|\leq C_M |x_0-y|^s B^{-1}\left(  \frac{1}{|x_0-y|^n} \iint_{\mathbb{R}^{2n}} A(|D^s u_\alpha|) \,d\nu_n\right).
$$
Using expression \eqref{eq} and  item (ii) of Lemma \ref{lema1}, since $\Psi_s(r)\approx r^s B^{-1}(r^{-n})$ for all $r>0$, there exists $c_1>0$ depending only on $n$ and $s$ such that
\begin{align}\label{desig0}
\begin{split}
|u_\alpha(x_0)|&\leq C_M |x_0-y|^s B^{-1}\left(  \frac{\lambda_{1,\alpha} }{|x_0-y|^n}  \int_\Omega \omega A(|u_\alpha|)\,dx\right)\\
&= C_M |x_0-y|^s B^{-1}\left(   \frac{ \alpha \lambda_{1,\alpha}}{|x_0-y|^n}\right)\\
&=
C_M (\alpha\lambda_{1,\alpha})^\frac{s}{n} \left((\alpha \lambda_{1,\alpha})^{-\frac{1}{n}} |x_0-y| \right)^s B^{-1}\left(  \left( (\alpha \lambda_{1,\alpha}^{-\frac1n} |x_0-y| \right)^{-n} \right)\\
&\leq  
c_1  C_M \left(\alpha \lambda_{1,\alpha}  \right)^{\frac{s}{n}}\Psi_s \left( |x_0-y| \left( \alpha \lambda_{1,\alpha}^s   \right)^{-\frac{1}{n}}\right).
\end{split}
\end{align}
Moreover, by definition of  inner radius we get that
$$
|x_0-y| \leq \max\{d(x,\partial\Omega)\colon x\in \Omega\} = r_\Omega.
$$
Hence, since   $\Psi_s$ is non-decreasing in light of Lemma \ref{lema1}, inequality \eqref{desig0} yields
\begin{align} \label{eq.u0}
|u_\alpha(x_0)|&\leq c_1 C_M \left( \alpha \lambda_{1,\alpha} \right)^\frac{s}{n} \Psi_s \left(   r_\Omega \left( \alpha \lambda_{1,\alpha}\right)^{-\frac{1}{n}}  \right).
\end{align}
\noindent From Lemma \ref{lema.e}, the function $E(\alpha):=\alpha \lambda_{1,\alpha}^s$ is strictly increasing, positive and Lipschitz continuous for $\alpha>0$, and satisfies that $E(0)=0$, $E(\infty)=\infty$.  Hence, defining $f(\alpha):=r_\Omega E(\alpha)^{-\frac{1}{n}}$, $\alpha>0$,  we get that $f$ is a  strictly decreasing continuous function such that $f(0):=\lim_{\alpha\to 0^+}f(\alpha)=\infty$ and $f(\infty):=\lim_{\alpha\to \infty} f(\alpha)=0$. From these properties there exists $\alpha_0>0$ such that $f(\alpha_0)=1$ and we have that:
$$
f(\alpha)> 1 \quad \text{ when } \alpha< \alpha_0, \qquad f(\alpha)<1 \quad \text{ when } \alpha>\alpha_0.
$$
In particular we have that $f(\alpha)> 1$ when $\alpha\ll1$ and $f(\alpha)<  1$ when $\alpha\gg1$.

\medskip

\noindent {\bf Case $\alpha>\alpha_0$.}
Since  $f(\alpha)< 1$, assuming that $i_\infty(A) >\frac{n}{s}$, by  of Lemma \ref{lema1} (iii) we get 
$$
\Psi_s(f(\alpha)) \approx f(\alpha)^s A^{-1} (f(\alpha)^{-n}).
$$
Then, there exists $c_2>0$ depending only on $n$ and $s$, and  \eqref{eq.u0} gives that
\begin{align*}
|u_\alpha(x_0)|&\leq 
c_1 C_M \left( \alpha \lambda_{1,\alpha} \right)^\frac{s}{n} \Psi_s \left(   r_\Omega \left( \alpha \lambda_{1,\alpha}\right)^{-\frac{1}{n}}  \right)\leq 
C r_\Omega ^s A^{-1} \left(   r_\Omega^{-n}\alpha \lambda_{1,\alpha} \right)
\end{align*} 
being $C=c_1 c_2 C_M$. Moreover, since $A$ is non-decreasing, the previous expression yields
\begin{align} \label{eq.aa}
\begin{split}
A(C^{-1} r_\Omega^{-s}     |u_\alpha(x_0)| )\leq
 r_\Omega^{-n} \alpha \lambda_{1,\alpha}&=
  r_\Omega^{-n} \lambda_{1,\alpha} \int_\Omega \omega A(|u_\alpha(x)|)\,dx\\
&\leq 
r_\Omega^{-n} \lambda_{1,\alpha} A(|u_\alpha(x_0)|) \|\omega\|_{L^1(\Omega)}.
\end{split}
\end{align}
As a consequence, equation \eqref{eq.aa} yields the following:
\begin{align*}
\frac{r_\Omega^n}{\|\omega\|_{L^1(\Omega)}} 
&\leq
\lambda_{1,\alpha}   \frac{A(|u_\alpha(x_0)|)}{A(C^{-1} r_\Omega^{-s}     |u_\alpha(x_0)| )}\leq \lambda_{1,\alpha}  \sup_{t>0}
  \frac{A(t)}{A(C^{-1} r_\Omega^{-s}     t )}\\
  &= \lambda_{1,\alpha}  \sup_{\tau >0}
  \frac{A(C r_\Omega^s \tau)}{A(\tau )} = \lambda_{1,\alpha}  M_A(C r_\Omega^s),
\end{align*}
where $M_A$ is the Matuszewska-Orlicz function associated to $A$ defined in Section \ref{sec.matu}. Since $M$ is submultiplicative, there exists $c>0$ depending on $A$ such that $M_A(Cr_\Omega^s)\leq cM_A(C)M_A(r_\Omega^s)$, and the inequality above leads to the following lower bound for $\lambda_{1,\alpha}$:
$$
 \frac{r_\Omega^n}{cM_A(C)|\omega\|_{L^1(\Omega)}} \frac{1}{M_A(r_\Omega^s)} \le \lambda_{1,\alpha} . 
$$

\medskip

\noindent {\bf Case $\alpha<\alpha_0$.} Here $f(\alpha) > 1$, then assuming  $i_0(A) >\frac{n}{s}$, by Lemma \ref{lema1} (iv) we get that
$$
\Psi_s(f(\alpha)) \approx   f(\alpha)^s A^{-1} (f(\alpha)^{-n}).
$$
Proceeding analogously as in the previous case we get the result.
\end{proof}

A similar argument to the proof of Theorem \ref{teo1.intro} yields a lower bound for the critical value, involving the inverse of $A$ instead of the Matuszewska-Orlicz function $M_A$.
 
\begin{thm}  \label{teo2}
Let $s\in (0,1)$, $\alpha>0$ and let $A$ be a Young function satisfying \eqref{cond1}. Given $\omega \in L^1(\Omega)$ consider the critical value $\lambda_{1,\alpha}^s$  defined in \eqref{minimi}. It holds that
\begin{itemize}
\item[i)] There exists a unique $\alpha_0>0$ satisfying the equation
$$
\alpha_0 \lambda_{1,\alpha_0}^s = r_\Omega^n.
$$

\item[ii)] 
Assume that $i_0(A)>\frac{n}{s}$ when $\alpha\leq \alpha_0$, or $i_\infty(A) >\frac{n}{s}$ when $\alpha> \alpha_0$.  Then, there exists a  $C>0$ depending only on  $s$, $n$ and $A$ such that
\begin{equation} \label{ineq.2}
\frac{r_\Omega^n}{\alpha} A\left( \frac{1}{C r_\Omega ^s}A^{-1}\left( \frac{\alpha}{ \|\omega\|_{L^1(\Omega)}} \right) \right)
\leq 
    \lambda_{1,\alpha}^s.
\end{equation}

\noindent In particular, this holds when $i_0(A)>\frac{n}{s}$ if $\alpha\ll1$ and when $i_\infty(A) >\frac{n}{s}$ if $\alpha\gg1$.
\end{itemize}
\end{thm}

\begin{proof}
Fixed $\alpha>0$, let $u_\alpha^s \in W^s_0 L^A(\Omega)$ be a minimizer of \eqref{minimi} such that $\int_\Omega \omega A(|u_\alpha^s|)\,dx =\alpha$, i.e., the pair $(u_\alpha^s, \lambda_{1,\alpha}^s)$ satisfies equation \eqref{eq}, where  $\lambda_{1,\alpha}^s$ is defined in \eqref{minimi}. Since $s\in(0,1)$ is fixed, for simplicity we will drop the dependence in $s$.

\noindent In light of Proposition \ref{morrey}  $u_\alpha$ is  continuous and hence there exists $x_0\in \Omega$ such that
$$
|u_\alpha(x_0)|=\max\{|u_\alpha(x)|\colon x \in \mathbb{R}^n\}>0.
$$
Since
$$
\alpha = \int_\Omega \omega A(|u_\alpha|)\,dx \leq A(|u_\alpha(x_0)|) \|\omega\|_{L^1(\Omega)},
$$
using expression \eqref{desig0} and the fact the $A$ is no decreasing,  we get that
\begin{align*}
\alpha &\leq
A\left( c_1  C_M \left(\alpha \lambda_{1,\alpha}  \right)^{\frac{s}{n}}\Psi_s \left( |x_0-y| \left( \alpha \lambda_{1,\alpha}   \right)^{-\frac{1}{n}}\right) \right) \|\omega\|_{L^1(\Omega)}
\end{align*}
for any $y\in \partial\Omega$, and therefore,  denoting by $r_\Omega$ the inner radius of $\Omega$, we get
\begin{equation} \label{desig11}
\alpha \|\omega\|_{L^1(\Omega)}^{-1}  
\leq
A\left( c_1  C_M \left(\alpha \lambda_{1,\alpha} \right)^{\frac{s}{n}}\Psi_s \left( r_\Omega \left( \alpha \lambda_{1,\alpha}   \right)^{-\frac{1}{n}}\right) \right).
\end{equation}
As in the proof of Theorem \ref{teo1.intro},  there exists $\alpha_0>0$ such that  $r_\Omega^n  \leq \alpha \lambda_{1,\alpha}$ when $\alpha>\alpha_0$ and $r_\Omega^n \geq \alpha \lambda_{1,\alpha}$ when $\alpha<\alpha_0$.

\medskip

{\bf Case }$\alpha>\alpha_0$. Assuming  $i_\infty(A) >\frac{n}{s}$, by  Lemma \ref{lema1} (iii) we get that $\Psi_s(t) \approx t^s A^{-1} (t^{-n})$ for any $t>0$. Then, there exists $c_2=c_2(n,s)>0$ for which  \eqref{desig11} yields
\begin{align*}
\alpha \|\omega\|_{L^1(\Omega)}^{-1}\leq 
A\left( C r_\Omega^s A^{-1} \left(   r_\Omega^{-n}\alpha \lambda_{1,\alpha} \right) \right)
\end{align*}
where $C=c_1c_2C_M$. Since $A$ is non decreasing, the previous expression gives that 
\begin{align*}
\frac{r_\Omega^n}{\alpha} A\left( \frac{1}{C r_\Omega ^s}A^{-1}\left( \frac{\alpha}{ \|\omega\|_{L^1(\Omega)}} \right) \right)
\leq 
    \lambda_{1,\alpha}.
\end{align*} 

{\bf Case }$\alpha>\alpha_0$. Assuming that $i_0(A) >\frac{n}{s}$, the bound follows analogously.
\end{proof}

As a direct consequence of  Theorem \ref{teo1.intro} and Lemma \ref{lema.comp} we get the following.
\begin{cor} \label{coro.1}
Under the assumptions and notation of Theorem \ref{teo1.intro}, if additionally $A\in \Delta_2$, then it holds that
$$
\frac{C}{p_A\|\omega\|_{L^1(\Omega)}} \frac{r_\Omega^n}{M( r_\Omega^s)} \le \Lambda_{1,\alpha}^s 
$$
where  $p_A=\sup_{\beta>0} \frac{a(\beta)\beta}{A(\beta)}$.
\end{cor}
\begin{proof}
It is direct from Theorem \ref{teo1.intro} by using  Lemma \ref{lema.comp}.
\end{proof}

\begin{proof}[Proof of Theorem \ref{teo2.intro}]
Fixed $\alpha>0$, let $u_\alpha^s \in W^s_0 L^A(\Omega)$ be a minimizer of \eqref{minimi}  such that $\int_\Omega \omega A(|u_\alpha^s|)\,dx =\alpha$, that is, the pair $(u_\alpha^s, \lambda_{1,\alpha}^s)$ satisfies  \eqref{eq},  where  $\lambda_{1,\alpha}^s$ is defined in \eqref{minimi}. 

Denote by $\text{d}_\Omega$ the diameter of $\Omega$. The Hardy inequality given in Proposition \ref{hardy} together with \eqref{eq} and the monotonicity of $A$ gives that
\begin{align} \label{eqh1} 
\begin{split}
\int_{\mathbb{R}^n} A\left(\frac{c_1 |u_\alpha(x)|}{\text{d}_\Omega^s}\right)dx &\leq \int_\Omega A\left(\frac{c_1|u_\alpha(x)|}{|x|^s}\right)dx\\
& \leq (1-s)\iint_{\mathbb{R}^{2n}} A(|D^s u_\alpha|) \,d\nu_n\\
& =(1-s) \lambda_{1,\alpha}^s \int_\Omega \omega A(|u_\alpha|)\,dx
\end{split}
\end{align}
where $c_1=C_{H_1}C_{H_2}^{-1}$, and  $C_{H_1}, C_{H_2}>0$ are the constants given in Proposition \ref{hardy}, which depend only on $n$ and $s$. Now, we compute the following inequality:
\begin{align} \label{eqh2}
\begin{split}
\int_\Omega A(|u_\alpha|)\,dx &= \int_\Omega \frac{A(|u_\alpha|)}{A(c_1 \text{d}_\Omega^{-s}|u_\alpha| ) } A(c_1 \text{d}_\Omega^{-s}|u_\alpha| )\,dx\\
&\leq 
\sup_{t\in (0, \|u\|_\infty)}  \frac{A(t)}{A(c_1 \text{d}_\Omega^{-s} t)}   \int_\Omega A(c_1 \text{d}_\Omega^{-s}|u_\alpha| )\,dx\\
&\leq 
\sup_{\tau >0} \frac{A(c^{-1} d_\Omega^s \tau)}{A(\tau)} \int_\Omega A(c_1 \text{d}_\Omega^{-s}|u_\alpha| )\,dx\\
&= M_A(c^{-1} d_\Omega^s) \int_\Omega A\left(\frac{c_1 |u_\alpha|}{d_\Omega^s} \right)\,dx.
\end{split}
\end{align}
From \eqref{eqh1}, \eqref{eqh2} and the fact that $M_A$ is submultiplicative, we get that
\begin{align*}
1&\leq (1-s)\lambda_{1,\alpha}^s \|\omega\|_{L^\infty(\Omega)}  M_A(c^{-1} d_\Omega^s)\\ 
&\leq 
(1-s)\lambda_{1,\alpha}^s \|\omega\|_{L^\infty(\Omega)}  \tilde c M_A(c^{-1}) M_A(d_\Omega^s)
\end{align*}
with $\tilde c= \tilde c(c,A)$, which concludes the proof.
\end{proof}

As a direct consequence of  Theorem \ref{teo2.intro} and Lemma \ref{lema.comp} we get the following.
\begin{cor} \label{coro.2}
Under the assumptions and notation of Theorem \ref{teo2.intro}, if additionally $A\in \Delta_2$, then it holds that
$$
\frac{C}{p_A\|\omega\|_{L^\infty(\Omega)} M_A(C d_\Omega^s) } \le \Lambda_{1,\alpha}^s
$$
where  $p_A=\sup_{\beta>0} \frac{a(\beta)\beta}{A(\beta)}$.
\end{cor}

When $A\in \Delta_2$ we can improve Theorem \ref{teo2.intro} by replacing $d_\Omega$ with $r_\Omega$.

\begin{proof}[Proof of Theorem \ref{teo4.intro}]
The proof is analogous to that of Theorem \ref{teo2.intro}, noting that in \eqref{eqh1}, the Hardy inequality stated in Proposition \ref{hardy2}, together with \eqref{eq}, leads to
\begin{align*}  
\begin{split}
\frac{1}{C}\int_{\mathbb{R}^n} A\left(\frac{ |u_\alpha(x)|}{r_\Omega^s}\right)\,dx &\leq   \int_\Omega A\left(\frac{|u_\alpha(x)|}{\delta_\Omega(x)^s}\right)\,dx\\
&\leq \iint_{\mathbb{R}^{2n}} A(|D^s u_\alpha|) \,d\nu_n = \lambda_{1,\alpha}^s \int_\Omega \omega A(|u_\alpha|)\,dx,
\end{split}
\end{align*}
where $\delta_\Omega(x)$ denotes the distance from $x$ to $\partial \Omega$,  giving the desired result.
\end{proof}

\subsection*{Data availability}
The authors state that no data were generated or analyzed in this paper, so data availability is not applicable.

\subsection*{Statements and Declarations}
The authors have no conflicts of interest to declare that are relevant to the content of this article.


\end{document}